\documentclass{article}
\usepackage{amsmath,amsfonts,amssymb,amsthm,epsfig,epstopdf,titling,url,array}
\usepackage{eqnarray}
\usepackage{enumerate}
\usepackage[a4paper]{geometry}
\usepackage{amscd}
\usepackage{centernot}
\usepackage{authblk}
\theoremstyle{plain} 
\newtheorem{thm}{ Theorem}[section]
\newtheorem{lem}[thm]{Lemma}
\newtheorem{prop}[thm]{Proposition}
\newtheorem{cor}[thm]{Corollary}
\newtheorem{defn}[thm]{Definition}
\newtheorem{ex}[thm]{Example}

\newtheorem* {note*}{Note}

\newtheorem{qu}{\bf Question}

\newcommand{\iy} {\infty}
\newcommand{\N} {\mathbb{N}}

\newcommand{\Z} {\mathbb{Z}}

\newcommand{\lm}{\lambda}

\newcommand{\de}{\delta}
\newcommand{\eps}{\epsilon}

\newcommand{\h}{\mathcal H}

\newcommand{\F}{\mathcal F}
\newcommand{\G}{\mathcal G}

\newcommand{\BH}{\mathcal B(\mathcal H)}

\newcommand{\norm}[1]{\left\Vert#1\right\Vert}

\newcommand{\set}[1]{\left\{#1\right\}}

\newcommand{\brc}[1]{\left(#1\right)}

\newcommand{\LZ}{\ell^{2}(\mathbb Z)}
\newcommand{\LN}{\ell^{2}(\mathbb N)}
\newcommand{\LNp}{\ell^{p}(\mathbb N)}
\newcommand{\m}{\mathcal {M}}

\title{\bf \Large Subspace-hypercyclic weighted shifts}
\bf
\author[1]{\bf\footnotesize Nareen Bamerni \thanks{nareen\_bamerni@yahoo.com}}
\author[2]{\bf Adem K{\i}l{\i}\c{c}man \thanks{akilicman@yahoo.com}}
\affil[1,2]{\bf Department of Mathematics, University Putra Malaysia,
43400 UPM, Serdang, Selangor, Malaysia}

\begin{document}
\date{}
\maketitle
\begin{abstract}
Our aim in this paper is to obtain necessary and sufficient conditions for weighted shift operators on the Hilbert spaces $\LZ$ and $\LN$ to be subspace-transitive, consequently,
 we show that the Herrero question \cite{limits} holds true even on a subspace of a Hilbert space, i.e. there exists an operator $T$ such that both $T$ and $T^*$ are subspace-hypercyclic operators for some subspaces. We display the conditions on the direct sum of two invertable bilateral forward weighted shift operators to be subspace-hypercyclic.
\end{abstract}
\section{Introduction}
A bounded linear operator $T$ on a separable Hilbert space $\h$ is hypercyclic if there is a vector $x\in \h$ such that $Orb(T,x)=\set{T^nx:n\ge 0}$ is dense in $\h$, such a vector $x$ is called  hypercyclic for $T$. The first example of a hypercyclic operator on a Banach space was constructed by Rolewicz in 1969 \cite{Rolewicz}. He showed that if $B$ 
is the backward shift on $\LNp$ then $\lm B$ is hypercyclic if and only if $|\lm|> 1$. The hypercyclicity concept was probabely born with the thesis of Kitai in 1982 \cite{Kitai} who introduced the hypercyclic criterion to ensure the existence of hypercyclic operators. For more information on hypercyclic operators we refer the reader to \cite{dynamic, Erdman}. \\

In 1991, Herrero \cite{limits} asked for the existence of a hypercyclic operator $T$ such that its adjoint is also hypercyclic, in the same year, Sales \cite{3} constructed such an example.\\

In 1995, Salas \cite{4} gave necessary and sufficient conditions for weighted shift operators to be hypercyclic or supercyclic and consequently, he found another example supporting Herrero's question.  However, that characterization was so complicated; therefore, Feldman \cite{H C Bi.} found some simpler conditions for the invertable shift case and consequently, he showed that the same conditions hold also for non invertable case but by considering that the negative or positive weights are bounded below. \\

In 2011,  Madore and Mart\'{i}nez-Avenda\~{n}o \cite{sub hyp} studied the density of the orbit in a non-trivial subspace instead of the whole space and called that phenomenon the subspace-hypercyclicity. For the series of references on subspaces-hypercyclic operators see \cite{Some ques, C.M, sub hyp, notes on sub}

\begin{defn}\cite{sub hyp}
Let $T\in \BH$ and $\m$ be a closed subspace of $\h$. Then $T$ is called $\m$-hypercyclic or subspace-hypercyclic operator for a subspace $\m$  if there exists a vector $x\in \h$ such that $Orb(T,x)\cap \m $ is dense in $\m$. We call $x$ an $\m$-hypercyclic vector for $T$.
\end{defn}

\begin{defn}\cite{sub hyp}
Let $T \in \BH$ and $\m$ be a closed subspace of $\h$. Then $T$ is called $\m$-transitive or subspace-transitive for a subspace $\m$ if for each pair of non-empty open sets $U_1,\,U_2$ of $\m$ there exists an $n \in \N$ such that $T^{-n}U_1 \cap U_2$ contains a non-empty relatively open set in $\m$.
\end{defn}

\begin{thm}\cite{sub hyp}\label{MTD}
Every $\m$-transitive operator on $\h$ is $\m$-hypercyclic. 
\end{thm}

\begin{prop}\label{chMD}\cite{m1}
Let $T\in\BH$. The following statements are equivalent:
\begin{enumerate}
 	\item $T$ is $\m$-transitive.
  \item for all $x,y\in \m$, there exist sequeces $\set{x_k}\in \m$ and $\set{n_k}\in \N$  such that for all $k\ge 1$, $T^{n_k}\m\subseteq \m$ and as $k\to \iy$, $\set{x_k}\to x$ and $T^{n_k}x_k\to y$.
  \item for each $x,y\in\m$ and each neighborhood $W$ for
  zero in $\m$, there exist $z\in\m$ and $n\in\N$;
  such that $x-z\in W$, $T^n z-y\in W$ and $T^n\m\subseteq \m$.
\end{enumerate}
\end{prop}

In the present paper, we follow the line of Sales's proofs \cite{4} and Feldman's proofs \cite{H C Bi.} to give the main results of this paper. Throught Theorem  \ref{forw} (Proposition \ref{bac}), we display necessary and sufficient conditions for an invertable bilateral forward (backward) weighted shift on a Hilbert space $\LZ$ to be a subspace-transitive operator. Then, Theorem \ref{84} (Proposition \ref{85}) shows that the same conditions still hold for non invertable bilateral forward (backward) weighted shifts but under an extra condition. Consequently, we show that the Herrero question \cite{limits} still holds for subspace-hypercyclic operators; i.e, there is an operator $T$ such that both $T$ and its adjoint are subspace-hypercyclic for some subspaces; however, we dont know whether they are subspace-hypercyclic for the same subspace or not. Also, we  characterize the direct sum of bilateral forward weighted shifts that are subspace-hypercyclic in terms of their weighte sequences. \\ 
Finally, we characterize unilateral backward weighted shift operators that are subspace-hypercyclic in term of their weight sequences.
\section{Main results}
All subspaces $\m$ in this section are supposed to be topologically closed. We will suppose that $$\F=\set{m_r: r\in \N \mbox { and } e_{m_r}\in \mbox { Schauder basis for } \m} $$ 

 The next two Theorems give necessary and sufficient conditions for a bilateral weighted shift operator on the Hilbert space $\LZ$ to be $\m$-transitive. \\

Let $T$ be the bilateral forward weighted shift operator with a weigh sequence $\set{w_n}$, then  
$T(e_{r})=w_{r}e_{r+1}$ for all $r\in \Z$. Let $S$ be the right inverse (backward shift) to $T$ and be defined as follows: $S(e_{r})=\frac{1}{w_{r-1}}e_{r-1}$. Observe that $TSe_{r}=e_{r}$ for all $r\in \Z$. If $T$ is invertable then $T^{-1}=S$. 

$$\displaystyle T^k(e_{m_r})=\brc{\prod_{j=m_r}^{m_r+k-1}w_j}e_{m_r+k} \mbox{	and } \displaystyle S^k(e_{m_r})=\brc{\prod_{j=m_r-1}^{m_r-k}\frac{1}{w_{j}}}e_{m_r-k}$$
The proof of some results in this section follow the lines of \cite{4} and \cite{H C Bi.}.

\begin{thm} \label{19}
Let $T$ be a bilateral forward weighted shift in the Hilbert space $\LZ$ with a positive weight sequence
$\set{w_n}_{n\in\Z}$ and $\m$ be a subspace of $\LZ$. If for a given $\delta>0$ and $q\in \N$, there exist an arbitrary large positive number $n$ such that $T^n\m \subseteq \m$.
Then $T$ is $\m$-transitive if and only if for all  $m_j\in \F$ with $|m_j|\le q$

$$\displaystyle\prod_{k=m_j}^{m_j+n-1}w_k<\delta \mbox{	and } \prod_{k=1-m_j}^{n-m_j}\frac{1}{w_{-k}}<\delta$$
\end{thm}
\begin{proof}
Let $T$ be an $\m$-transitive operator, then by Proposition \ref{chMD} there exist a vector $x\in \m$ such that for $0<\eps<1$
\begin{eqnarray}\label{15}
\norm{x-\sum_{\stackrel{|m_j|\le q}{m_j\in\F }}e_{m_j}}<\eps
\end{eqnarray}
Also one can find a large $n$; $n>2q$, such that $T^n\m \subseteq \m$ and  
\begin{eqnarray}\label{16}
\norm{T^nx-\sum_{\stackrel{|m_j|\le q}{m_j\in\F }}e_{m_j}}<\eps
\end{eqnarray}
Inquality (\ref{15}) implies that $|x_{m_j}|>1-\eps$ if $|m_j|\le q$ and $|x_{m_j}|<\eps$ otherwise. Since $n>2q$ inquality (\ref{16}) implies that for  $|m_j|\le q$
$$\displaystyle|x_{m_j}|\norm{T^ne_{m_j}}=|x_{m_j}|\displaystyle \left(\prod_{k=0}^{n-1}w_{k+m_j}\right)<\eps $$
It follows that
 \begin{eqnarray}\label{17}
 \left(\prod_{k=0}^{n-1}w_{k+m_j}\right)< \frac{\eps}{|x_{m_j}|}<\frac{\eps}{1-\eps}<\delta
\end{eqnarray}
Also inquality (\ref{16}) implies that $\norm{x_{m_j-n}(T^ne_{m_j-n})-e_{m_j}}< \eps$ for $|m_j|\le q$. Thus $$\left|x_{m_j-n}\right| \left|\displaystyle \prod_{k=0}^{n-1}w_{m_j-n+k}-1\right|=\left|x_{m_j-n}\right|\left|\displaystyle \prod_{k=1}^{n}w_{m_j-k}-1\right|< \eps$$
Therefore
\begin{eqnarray}\label{18}
 \left(\prod_{k=1}^{n}w_{m_j-k}\right)>\frac{1-\eps}{\left|x_{m_j-n}\right|}>\frac{1-\eps}{\eps}>\frac{1}{\delta}
\end{eqnarray}
The proof follows by equation (\ref{17}) and (\ref{18}).\\

Conversely, let $\delta>0$, $q\in \N$, and $n$ be an arbitrary large positive number $n$ such that $T^n\m \subseteq \m$ and

$$\displaystyle\prod_{k=m_j}^{m_j+n-1}w_k<\delta \mbox{	and } \prod_{k=1-m_j}^{n-m_j}\frac{1}{w_{-k}}<\delta$$
for all $m_j\in \F$ with $|m_j|\le q$. Let $x=\sum_{\stackrel{|m_j|\le q}{m_j\in\F }}x_je_{m_j}$ and $y=\sum_{\stackrel{|m_j|\le q}{m_j\in\F }}y_je_{m_j}$ be two vectors in the span of $\set{e_{m_j}:|m_j|\le q}$. Then 

$$\norm{T^nx}\le \set{\prod_{k=m_j}^{m_j+n-1}w_k:|m_j|\le q}\norm{x}$$
and 
$$\norm{S^ny}\le \set{\prod_{k=1-m_j}^{n-m_j}\frac{1}{w_{-k}}:|m_j|\le q}\norm{y}$$
It is not difficult to apply \cite[Lemma 2 .2.]{4} for the subspace-transitivity case. Thus one can argue that $T$ is $\m$-transitive.
\end{proof}
\begin{lem}\label{35}
Let $T$ be an invertable bilateral weighted shift, $\set{n_k}$ be an increasing sequence of positive integers such that $n_k \to \iy$ as $k\to \iy$ and $T^{n_k}\m \subseteq \m$ for all $k\in \N$. If there exists an $m_i\in \F$ such that $T^{n_k}e_{m_i}\to 0$ as $k\to \iy$, then $T^{n_k}e_{m_r}\to 0$ for all $m_r\in \F$. 
\end{lem}
\begin{proof}
Since $T^{n_k}\m \subseteq \m$, the proof is analogous to the proof of \cite[Lemma 3.1.]{H C Bi.}.
\end{proof}
\begin{thm} \label{forw}
Let $T$ be an invertable bilateral forward weighted shift in the Hilbert space $\LZ$ with a positive weight sequence $\set{w_n}_{n\in\Z}$ and  $\m$ be a subspace of $\LZ$. Let $\set{n_k}$ be an increasing sequence of positive integers such that $n_k \to \iy$ as $k\to \iy$ and $T^{n_k}\m \subseteq \m$ for all $k\in \N$. Then $T$ is $\m$-transitive if and only if there exists $m_i\in \F$ such that
  	\begin{eqnarray}\label{65}  	
		\displaystyle\lim_{k\to \iy}\prod_{j=m_i}^{m_i+n_k-1}w_j=0 \mbox{	and } \displaystyle\lim_{k\to \iy}\prod_{j=1-m_i}^{n_k-m_i}\frac{1}{w_{-j}}=0
		\end{eqnarray}
\end{thm}

\begin{proof}
To Prove the ``if'' part , we will verify the $\m$-hypercyclic criterion with $D=D_1=D_2$ be a dense subset of $\m$ consisting of sequences that only have a finite number of non-zero entries. Let $x\in D$, then it is enough to show that $x=e_{m_r}$; $m_r\in \F$. In addition, by Lemma \ref{35} it suffices to show that $T^{n_k}e_{m_i} \to 0$ and $S^{n_k}e_{m_i} \to 0$. However, that is clear because $\displaystyle \norm{T^{n_k}e_{m_i}}=\prod_{j=m_i}^{m_i+n_k-1}w_j\to 0$ and $\displaystyle \norm{S^{n_k}e_{m_i}}=\prod_{j=m_i-1}^{m_i-n_k}\frac{1}{w_{j}}\to 0$. Moreover it is clear that $T^{n_k}S^{n_k}x=x$. By taking $x_k=S^{n_k}x$, it is clear that the conditions of $\m$-hypercyclic criterion are satisfied.\\
The proof of the ``only if''  part follows from Theorem \ref{19}
\end{proof}
It is well known that a weighted shift operator $T$ is invertable if and only if there exists $b>0$ such that $|w_n|\ge b$ for all $n\in \Z$. The next theorem shows that the above theorem still holds by assuming that there exists $b>0$ such that $w_n\ge b$ for all $n<0$. 
\begin{thm} \label{84}
Let $T$ be a bilateral forward weighted shift in the Hilbert space $\LZ$ with a positive weight sequence
$\set{w_n}_{n\in\Z}$, $w_n \ge b>0$ for all $n<0$ and $\m$ be a subspace of $\LZ$. Let $\set{n_k}$ be an increasing sequence of positive integers such that $n_k \to \iy$ as $k\to \iy$ and $T^{n_k}\m \subseteq \m$ for all $k\in \N$. Suppose that there exist $m_i\in \F$ such that $m_i\ge 0$, then $T$ is $\m$-transitive if and only if 
  	\begin{eqnarray} 	
		\displaystyle\lim_{k\to \iy}\prod_{j=m_i}^{m_i+n_k-1}w_j=0 \mbox{	and } \displaystyle\lim_{k\to \iy}\prod_{j=1-m_i}^{n_k-m_i}\frac{1}{w_{-j}}=0
		\end{eqnarray}
\end{thm}
\begin{proof}
For ``if'' part, we will verify Theorem \ref{19}. Let $\eps >0$, $q\in \N$ and let $\de >0$; by hypothesis there exists an arbitrary large $n_k\in \set{n_k}$ such that $T^{n_k}\m \subseteq \m$ and $m_i\in \F$ such that $m_i\ge 0$ and
$$\displaystyle \prod_{j=m_i}^{m_i+n_k-1}w_j < \de \mbox{	and } \displaystyle \prod_{j=1-m_i}^{n_k-m_i}\frac{1}{w_{-j}}< \de$$ 
Let $n=n_k+m_i+q+1$ (which ensure that $m_p+n-1 \ge n_k+m_i$ for all $|m_p|\le q$). Now, for all $m_p\in \F$ with $|m_p|\le q$ we have 
\begin{eqnarray*}
\displaystyle \prod_{j=m_p}^{n+m_p-1}w_j &=&\left(\prod_{j=m_i}^{m_p-1}\frac{1}{w_j}\right) \left(\prod_{j=m_i}^{m_p-1}w_j \right) \left(\prod_{j=m_p}^{m_i+n_k-1}w_j\right) \left(\prod_{j=m_i+n_k}^{n+m_p-1}w_j \right)\\
 &=&\left(\prod_{j=m_i}^{m_p-1}\frac{1}{w_j}\right)  \left(\prod_{j=m_i}^{m_i+n_k-1}w_j\right) \left(\prod_{j=m_i+n_k}^{n+m_p-1}w_j \right)\\
 &\le & C \left(\prod_{j=m_i}^{m_i+n_k-1}w_j\right) \norm{T^2q}\\
 &\le & C \de  \norm{T^2q}.
\end{eqnarray*}
where $C$ is a constant depending only on $q$. So, if $\de < \frac{\eps}{C \norm{T^2q}}$, then  $\displaystyle \prod_{j=m_p}^{n+m_p-1}w_j \le \eps$ for all $m_p\in \F$ with $|m_p|\le q$.\\

Also, with the same choice of $n$, the condition $m_i \ge 0$ (which ensure that $n-m_p > n_k-m_i+1$) and $|m_p|\le q$ we have
\begin{eqnarray*}
\displaystyle \prod_{j=1-m_p}^{n-m_p}\frac{1}{w_{-j}}&=& \left(\prod_{j=1-m_i}^{-m_p}w_{-j}\right) \left(\prod_{j=1-m_i}^{-m_p}\frac{1}{w_{-j}}\right) \left( \prod_{j=1-m_p}^{n_k-m_i}\frac{1}{w_{-j}}\right) \left(\prod_{j=n_k-m_i+1}^{n-m_p}\frac{1}{w_{-j}}\right)\\
&=& \left(\prod_{j=1-m_i}^{-m_p}w_{-j}\right) \left( \prod_{j=1-m_i}^{n_k-m_i}\frac{1}{w_{-j}}\right) \left(\prod_{j=n_k-m_i+1}^{n-m_p}\frac{1}{w_{-j}}\right)\\
&\le & L \de \left(\frac{1}{b}\right)^{2q}\\
\end{eqnarray*}
where $b$ is a lower bound for the negative weights and $L$ is a constant depending only on $q$. Hence, if $\de < \frac{b^{2q}\eps}{L}$, then $\displaystyle \prod_{j=1-m_p}^{n-m_p}\frac{1}{w_{-j}}\le \eps$ for all $m_p\in \F$ with $|m_p|\le q$.\\

The converse side follows immediately by Theorem \ref{19}. 
\end{proof}
The following Theorem can be proved by the same arguments for proving Theorem \ref{forw}; therefore, we will state it without proof. \\ 
Let $\m_1$ and $\m_2$ be closed subspaces of the Hilbert space $\LZ$. Suppose that $$\F_1=\left\{m_r: r\in \N \mbox { and } e_{m_r}\in \mbox { Schauder basis for } \m_1 \right\} $$ and  $$\F_2=\left\{h_r: r\in \N \mbox { and } e_{h_r}\in \mbox { Schauder basis for } \m_2\right\} $$
\begin{thm}\label{28}
Let $T_1$ and $T_2$ be invertable bilateral forward weighted shifts in the Hilbert space $\LZ$ with a positive weight sequence
$\set{w_n}_{n\in\Z}$ and $\set{a_n}_{n\in\Z}$, respectively. Let $\m_1$ and $\m_2$ be subspaces of $\LZ$. Let $\set{n_k}$ be an increasing sequence of positive integers such that $n_k \to \iy$ as $k\to \iy$ and $(T_1 \oplus T_2)^{n_k}(\m_1 \oplus \m_2)\subseteq \m_1 \oplus \m_2$ for all $k\in \N$. Then $T_1 \oplus T_2$ is $\m_1 \oplus \m_2$-transitive if and only if there exist $m_i \in \F_1$ and $h_p\in \F_2$ such that
\begin{eqnarray}\label{25}
  	\displaystyle\lim_{k\to \iy} \max \left\{ \prod_{j=m_i}^{m_i+n_k-1}w_j , \prod_{j=h_p}^{h_p+n_k-1}a_j  \right\}=0 
\end{eqnarray}		
		and
\begin{eqnarray}\label{26}		
		\displaystyle\lim_{k\to \iy} \max \left\{\prod_{j=1-m_i}^{n_k-m_i}\frac{1}{w_{-j}} ,  \prod_{j=1-h_p}^{n_k-h_p}\frac{1}{a_{-j}}\right\}=0
		\end{eqnarray}
\end{thm}
It can be easily shown that the above theorem does not hold just for two operators but for a finite number of invertable bilateral forward weighted shifts.\\
By the same way we can characterize the $\m$-hypercyclic backward weighted shifts since they are unitarily equivalent to forward shifts.
\begin{prop}\label{bac}
Let $T$ be an invertable bilateral backward weighted shift in the Hilbert space $\LZ$ with a positive weight sequence $\set{w_n}_{n\in\Z}$ and  $\m$ be a subspace of $\LZ$. Let $\set{n_k}$ be an increasing sequence of positive integers such that $n_k \to \iy$ as $k\to \iy$ and $T^{n_k}\m \subseteq \m$ for all $k\in \N$. Then $T$ is $\m$-transitive if and only if there exists $m_i\in \F$ such that
  			$$\displaystyle\lim_{k\to \iy}\prod_{j=m_i}^{m_i+n_k-1}w_{-j}=0 \mbox{	and } \displaystyle\lim_{k\to \iy}\prod_{j=1-m_i}^{n_k-m_i}\frac{1}{w_{j}}=0$$
		
\end{prop}
\begin{prop} \label{85}
Let $T$ be a bilateral backward weighted shift in the Hilbert space $\LZ$ with a positive weight sequence
$\set{w_n}_{n\in\Z}$, $w_n \ge b>0$ for all $n<0$ and $\m$ be a subspace of $\LZ$. Let $\set{n_k}$ be an increasing sequence of positive integers such that $n_k \to \iy$ as $k\to \iy$ and $T^{n_k}\m \subseteq \m$ for all $k\in \N$. Suppose that there exist $m_i\in \F$ such that $m_i\ge 0$, then $T$ is $\m$-transitive if and only if 
	$$\displaystyle\lim_{k\to \iy}\prod_{j=m_i}^{m_i+n_k-1}w_{-j}=0 \mbox{	and } \displaystyle\lim_{k\to \iy}\prod_{j=1-m_i}^{n_k-m_i}\frac{1}{w_{j}}=0$$
\end{prop}
The following example shows that the Herrero question \cite{limits} holds true even on a subspace of a Hilbert space.
\begin{ex}\label{adjoint} 
There exists an operator $T$ such that both $T$ and $T^*$ are subspace-hypercyclic for some subspaces.  
\end{ex}
\begin{proof}
One can construct a weight sequence $\set{w_n}$ such that $\set{w_n}$ satisfies the conditions of Theorem \ref{forw} for a subspace $\m_1$ and satisfies  the conditions of Proposition \ref{bac} for a subspace $\m_2$. If we set $T$ to have the weight sequence $\set{w_n}$, it immediately follows that both $T$ and $T^*$ are subspaces-transitive operators for $\m_1$ and $\m_2$ respectively .   
\end{proof}
Since $T^n(\m)\subseteq \m$ if and only if $T^{*n}(\m^\perp)\subseteq \m^\perp$, then one may conjecture that $T$ is $\m$-transitive if and only if $T^*$ is $\m^\perp$-transitive. However, that is not true
\begin{ex}
Let $B$ be a unilateral backward shift operator and 
$$\m=\{\set{x_n}_{n=0}^\iy:x_{2n}=0 \mbox{ for all } n\in \N\} \subset \LN$$
Then $2B$ is $\m$-hypercyclic see (\cite[Example 3.7.]{sub hyp}). However $(2B)^*=2F$, where $F$ is the unilateral forward shift, which can not be $\m^\perp$-subspace, since unilateral forward shift can never be subspace-hypercyclic for any subspaces.
\end{ex}
\begin{qu}
If $T$ is $\m_1$-transitive and $T^*$ is $\m_2$-transitive, Is there any relation between $\m_2$ and $\m_1^\perp$?
\end{qu} 
We now turn to the unilateral weighted shift operators acting on $\LN$. Let $B$ be a unilateral backward weighted shift operator with a positive weight sequence $\set{w_n:n\in \N}$ then $B$ is defined by $B e_0=0$ and $B e_n=w_ne_{n-1}$ for all $n\ge 1$, and let $F$ be a unilateral forward weighted shift operator with a positive weight sequence $\set{w_n:n\in \N}$ then $F$ is defined by $F e_n=w_ne_{n+1}$ for all $n\ge 0$. We will suppose that $\G=\left\{m_r: r\in \N \mbox { and } e_{m_r}\in \mbox { Schauder basis for } \m \right\} $.\\

 It is clear that unilateral forward weighted shifts can not be subspace-hypercyclic operators for any subspaces; therefore,  we will characterize only the unilateral backward weighted shifts that are subspace-hypercyclic operators for some subspaces.
\begin{thm}
Let $B$ be a unilateral backward weighted shift operator on $\LN$ with positive weight sequence $\set{w_n:n\in \N}$ and $\m$ be a subspace of $\LN$. Let $\set{n_k}$ be an increasing sequence of positive integers such that $\set{n_k}\to \iy$ as $k\to \iy$ and $B^{n_k}\m \subseteq \m$ for all $k\in \N$. Then $B$ satisfies $\m$-hyperccyclic criterion with respect to the sequence $\set{n_k}$ if and only if there exists an $m_i \in \G$ such that
$$\lim \sup_{n\to \iy}(w_{m_i+1}w_{m_i+2}\ldots w_{m_i+n})\to \iy$$
\end{thm}
\begin{proof}
For the ``if'' part, let $\set{n_k}$ be an increasing sequence of positive integers such that $\set{n_k}\to \iy$ as $k\to \iy$ and $B^{n_k}\m \subseteq \m$ for all $k\in \N$, and let
$D=D_1=D_2$ made up of all finitely supported sequences subsets of $\m$, then  it is clear that $D$ is a dense subset of $\m$. Then for all $x\in D$, $B^{n_k}x=0$, since $x$ has only finite numbers of nonzero elements. Let $x_k=1/({w_{m_i+1}w_{m_i+2}\ldots w_{m_i+n_k}})F^{n_k}y\subseteq \m$. By hypothesis, $\norm{x_k}=1/({|w_{m_i+1}w_{m_i+2}\ldots w_{m_i+n_k}|}) \left\|F^{n_k}y\right\| \to 0$ as $k\to \iy$,  hence the result. \\

For the ``only if'' part, suppose that $\set{n_k}$ is an increasing sequence of positive integers such that $\set{n_k}\to \iy$ as $k\to \iy$ and $B^{n_k}\m \subseteq \m$ for all $k\in \N$, and suppose that $B$ satisfies $\m$-hyperccyclic criterion with respect to the sequence $\set{n_k}$ and $m_i\in \G$ such that $\lim \sup_{n\to \iy}(w_{m_i+1}w_{m_i+2}\ldots w_{m_i+n})\to \iy$. Since the set of all $\m$-hypercyclic operators are dense in $\m$, then by Proposition \ref{chMD} one may find $x\in \m$ and a nonzero $n \in \N$ such that for any $0<\eps<1$
$$\norm{x-e_{m_i}}\le \eps \mbox { and }  \norm{B^nx-e_{m_i}}\le \eps$$
Suppose that $x=(x_0,x_1, \ldots)$. From the first inquality, it follows that 
\begin{eqnarray}\label{100}
|x_i|\le \eps \mbox{ for all } i\in \N; i\neq m_i
\end{eqnarray}
From the second inquality, it follows that 
\begin{eqnarray}\label{101}
|x_{n+m_i}w_{1+m_i}w_{2+m_i}\ldots w_{n+m_i}|\ge 1-\eps 
\end{eqnarray}
From Equations (\ref{100}) and (\ref{101}), we get  $w_{1+m_i}w_{2+m_i}\ldots w_{n+m_i}\ge \frac{1-\eps}{\eps}$. By setting $\eps \to 0$, we get the result.
\end{proof}
We will suppose that $$\G_1=\set{m_r: r\in \N, e_{m_r}\in \mbox {the Schauder bases for } \m_1}$$ and  $$\G_2=\left\{ h_r: r\in \N, e_{h_r}\in \mbox {the Schauder bases for } \m_2 \right\}$$ 
\begin{cor}
Let $B_1$ and $B_2$ be unilateral backward weighted shifts in the Hilbert space $\LN$ with a positive weight sequence
$\set{w_n}_{n\in\N}$ and $\set{a_n}_{n\in\N}$, respectively. Let $\m_1$ and $\m_2$ be subspaces of $\LN$, and let $\set{n_k}$ be an increasing sequence of positive integers such that $\set{n_k}\to \iy$ as $k\to \iy$ and $(B_1 \oplus B_2)^{n_k}(\m_1 \oplus \m_2)\subseteq \m_1 \oplus \m_2$ for all $k\in \N$. Then $B_1 \oplus B_2$ satisfies $\m_1 \oplus \m_2$-hypercyclic criterion with respect to $\set{n_k}$ if and only if there exists $m_i \in \G_1$ and $h_p\in \G_2$ such that 
$$\sup_{n\in \N}\set{\min{(a_{h_p+1}a_{h_p+2}\ldots a_{h_p+n}), (w_{m_i+1}w_{m_i+2}\ldots w_{m_i+n})}}=\iy$$
\end{cor}


\end{document}